\documentclass{amsart}

\usepackage[margin=1in]{geometry}
\usepackage{amsmath,amsthm,amssymb,amstext}
\usepackage{mathrsfs}
\usepackage{tikz}
\usepackage{hyperref}
\usepackage[alphabetic,backrefs]{amsrefs}

\newtheorem{theorem}{Theorem}
\newtheorem{letteredtheorem}{Theorem}

\newtheorem{lemma}[theorem]{Lemma}
\newtheorem{corollary}[theorem]{Corollary}
\newtheorem{proposition}[theorem]{Proposition}

\theoremstyle{remark}
\newtheorem{remark}[theorem]{Remark}

\theoremstyle{definition}
\newtheorem{definition}[theorem]{Definition}

\numberwithin{theorem}{section}

\hypersetup{
  colorlinks=true,
  linkcolor=blue,
  anchorcolor=blue,
  citecolor=blue
}

\hypersetup{
  pdftitle={Stability of Kernel Sheaves on Del Pezzo Surfaces},
  pdfauthor={Nick Rekuski}
  }
  
\begin{document}
  \title{
    Stability of Kernel Sheaves on Del Pezzo Surfaces
  }
  \author{
    Nick Rekuski
  }
  \date{
    August 14, 2023
  }
  \address{
    Department of Mathematics, Wayne State University, Detroit, MI 48202, USA
  }
  \email{
    rekuski@wayne.edu
  }
  \thanks{
    The author was partially supported by the NSF grant DMS-2101761 during preparation of this article.
    The author is supported by the U.S. Department of Energy, Office of Science, Basic Energy Sciences, under Award Number DE-SC-SC0022134.
    This work is also partially supported by an OVPR Postdoctoral Award at Wayne State University.
  }

  \begin{abstract}
    Using techniques from Bridgeland stability, we show the kernel sheaf associated to sufficiently positive Gieseker stable sheaf on a Del Pezzo surface is slope stable.
    This is the first effective stability result for kernel sheaves associated to higher rank sheaves on surfaces and the first stability result for kernel sheaves associated to Gieseker stable sheaves.
  \end{abstract}

  \maketitle

\section{Introduction}
  Suppose $X$ is a smooth, irreducible, projective surface.
  If $\mathscr{E}$ is a globally-generated, torsion-free sheaf on $X$ then we can define its kernel sheaf (or syzygy sheaf) as the kernel of the valuation map:
  \[
    0\to\mathscr{M}_{\mathscr{E}}\to H^{0}(\mathscr{E})\otimes\mathscr{O}_{X}\to\mathscr{E}\to 0.
  \]
  The sheaves arising via this construction have a variety of applications, so there has been interest to establish the stability of $\mathscr{M}_{\mathscr{E}}$.
  In the case of curves, Butler showed if $\mathscr{E}$ is slope stable and $\operatorname{deg}(\mathscr{E})>2\operatorname{genus}(X)\operatorname{rank}(\mathscr{E})$ then $\mathscr{M}_{\mathscr{E}}$ is slope stable~\cite{butler1994}*{theorem 1.2}.
  Furthermore, in the case of curves more recent results have improved this bound and considered kernel sheaves associated incomplete linear systems (see \cite{brambila2019} and references within).
  
  In dimensions $2$ and above, the author showed if $\operatorname{rank}(\mathscr{E})=1$ and $h^{0}(\mathscr{E})\gg 0$ (with an explicit bound on the number of global sections) then $\mathscr{M}_{\mathscr{E}}$ is slope stable~\cite{rekuski2023}*{Theorem A}.
  For specific classes of varieties, there are many results for kernel sheaves associated to line bundles with much better positivity bounds than the author's result ~\cites{flenner1984,trivedi2010,coanda2011,camere2012,caucci2021,mukherjee2022,torres2022, miro2023a, miro2023b}.
  However, the only stability result for kernel sheaves associated to higher rank sheaves on surfaces has an ineffective bound on the necessary positivity~\cite{basu2021}.
  
  In this article, we use techniques from Bridgeland stability to give an effective stability results for kernel sheaves associated to higher rank stable sheaves on Del Pezzo surfaces.
  This is the first effective result of this kind.
  Moreover, this is the first stability result for kernel sheaves associated to Gieseker stable sheaves (rather than the stronger notion of $\mu_{H}$-stable sheaves).

  \begin{letteredtheorem}[\ref{theorem:lazarsfeldMukaiOnDelPezzo}]
    \label{mainTheorem}
    Assume $X$ is a smooth Del Pezzo surface over an algebraically closed field (of arbitrary characteristic).
    Let $\mathscr{E}$ be a globally generated, torsion-free, $(-K_{X},-\frac{K_{X}}{2}$)-Gieseker stable (e.g. $\mathscr{E}$ is $\mu_{-K_{X}}$-stable---see Lemma \ref{lemma:implicationsAmongCoherentStability}) sheaf on $X$ with associated kernel sheaf $\mathscr{M}$.
    If the following bounds are satisfied:
    \begin{itemize}
      \item{
        $0<\operatorname{deg}_{-K_{X}}(\mathscr{E})$,
      }
      \item{
        $0<\operatorname{ch}_{2}(\mathscr{E})$, and
      }
      \item{
        $\displaystyle \operatorname{deg}_{-K_{X}}(\mathscr{E})^{2}-2\operatorname{ch}_{2}(\mathscr{E})\operatorname{rank}(\mathscr{E})< 2\frac{\operatorname{ch}_{2}(\mathscr{E})}{\operatorname{deg}_{-K_{X}}(\mathscr{E})}+\frac{1}{\operatorname{rank}(\mathscr{E})^{2}}.
        $
      }
    \end{itemize}
    then $\mathscr{M}$ is $\mu_{-K_{X}}$-stable.
  \end{letteredtheorem}
  
  A priori, the positivity assumptions on $\mathscr{E}$ may never hold.
  To this end, we show in Corollary \ref{corollary:effectiveVersion} that any sufficiently high twist of a torsion-free, slope stable sheaf by $-K_{X}$ satisfies the bounds of Theorem \ref{mainTheorem}.
  We discuss necessity of the assumed bounds in Remark \ref{example:needMoreThanJustPositivity}.
  
\subsection*{Outline}
  In Section \ref{section:stabilityCoherent} we recall well known facts and definitions related to slope stability and twisted Gieseker stability.
  In Section \ref{section:tiltStability} we recall tilt stability then discuss the wall and chamber structure on a family of tilt stability conditions.
  In Section \ref{section:stabilityOfSyzygyBundles} we prove Theorem \ref{mainTheorem}.
  We also use Theorem \ref{mainTheorem} to describe the largest tilt-destabilizing wall of $\mathscr{M}[1]$.
  
  We further outline the proof of Theorem \ref{mainTheorem}.
  The theorem broadly consists of two steps.
  First we show $\mathscr{M}_{\mathscr{E}}[1]$ is tilt stable using wall-crossing for tilt stability.
  Second, we show tilt stability of $\mathscr{M}_{\mathscr{E}}[1]$ implies slope stability of $\mathscr{M}_{\mathscr{E}}$.
  
  The first step of this proof has also been used in other recent results~\cites{bayer2018,kopper2020,feyzbakhsh2022}.
  In short, shifting the short exact sequence
  \[
    0\to\mathscr{M}_{\mathscr{E}}\to H^{0}(\mathscr{E})\otimes\mathscr{O}_{X}\to\mathscr{E}\to 0
  \]
  gives the distinguished triangle
  \[
    \mathscr{E}\to\mathscr{M}[1]\to(H^{0}(\mathscr{E})\otimes\mathscr{O}_{X})[1]\to\mathscr{E}[1]
  \]
  in $D^{b}(X)$.
  The positivity bounds on $\mathscr{E}$ and techniques from Bridgeland stability show $\mathscr{E}$ is $\mu_{\alpha,\beta}^{\mathrm{tilt}}$-stable for $(\beta,\alpha)$ in a certain region of the upper half plane.
  
  For the second step of the theorem, we show $\mu_{\alpha,\beta}^{\mathrm{tilt}}$-stability of $\mathscr{M}[1]$ constrains the Chern characters of any maximal destabilizing subsheaf of $\mathscr{M}$.
  In the case of Del Pezzo surfaces these constraints are enough to show $\mathscr{M}$ is $\mu_{H}$-stable.
  
\subsection*{Notation and Assumptions}
  Assume $X$ is a smooth, projective, surface over an algebraically closed field (of arbitrary characteristic).
  Fix an ample divisor $H$ and $\mathbb{R}$-divisor $D$ on $X$.
  We also assume $X$ satisfies Bogomolov's inequality (see Lemma \ref{lemma:bogomolovInequality}).
  In characteristic $0$ this assumption always holds and in characteristic $p$ it holds for all surfaces with Kodaira dimension at most $1$ that are not quasielliptic (see Remark \ref{remark:characteristicP}).
  In particular, Bogomolov's inequality holds for Del Pezzo surfaces over an algebraically closed field of any characteristic.
  
  The bounded derived category of $\operatorname{Coh}(X)$ will be written $D^{b}(X)$.
  We write $\mathscr{H}^{i}:D^{b}(X)\to\operatorname{Coh}(X)$ for the associated cohomology functors.
  We use script letters (i.e. $\mathscr{E}$, $\mathscr{F}$, $\mathscr{G}$) for coherent sheaves and capital letters (i.e. $E$, $F$, $G$) for chain complexes in $D^{b}(X)$.
  We use $\operatorname{ch}_{i}(\mathscr{E})$ for the $i$th Chern character of $\mathscr{E}$.
  In particular, we write the degree of $\mathscr{E}$ (with respect to the ample divisor $H$) by $\operatorname{deg}_{H}(\mathscr{E})=H\cdot\operatorname{ch}_{1}(\mathscr{E})$.
  Given a chain complex $E\in D^{b}(X)$, we write
  \[
    \operatorname{ch}_{i}(E)=\sum_{i\in\mathbb{Z}}(-1)^{i}\operatorname{ch}_{i}(\mathcal{H}^{i}(E))
  \]

\subsection*{Acknowledgments}
  The author is thankful to Rajesh Kulkarni and Xuqiang Qin for many useful discussions.
  The author is also thankful to Rosa M. Mir{\'o}-Roig for pointing out the reference~\cite{basu2021}.

\section{Stability for Coherent Sheaves}
  \label{section:stabilityCoherent}
  In this section we recall two notions of stability for coherent sheaves: $\mu_{H}$-stability (also called slope, Mumford, or Mumford-Takemoto stability) and $(H,D)$-Gieseker stability.
  The notion of $(H,D)$-Gieseker stability is a variant of usual Gieseker stability that appears naturally in the context of Bridgeland stability.
  Like Gieseker stability, $(H,D)$-Gieseker stability is weaker than $\mu_{H}$-stability but stronger than $\mu_{H}$-semistability.
  
  \subsection*{Slope Stability}
  We discuss $\mu_{H}$-stability and recall Bogomolov's inequality.
  
  \begin{definition}
    \label{def:twistedSlopeStable}
    Suppose $\mathscr{E}\in\operatorname{Coh}(X)$.
    \begin{enumerate}
      \item{
        We define
        \[
          \mu_{H}(\mathscr{E})=
            \begin{cases}
              \frac{\operatorname{deg}_{H}(\mathscr{E})}{\operatorname{rank}(E)}: & \operatorname{rank}(\mathscr{E})\neq 0 \\
              +\infty: & \operatorname{rank}(\mathscr{E})=0
            \end{cases}
        \]
        which is called the \textit{slope} of $\mathscr{E}$ with respect to $H$.
      }
      \item{
        We say nonzero $\mathscr{E}\in\operatorname{Coh}(X)$ is $\mu_{H}$-\textit{(semi)stable} if every proper nonzero subsheaf $0\to\mathscr{F}\to\mathscr{E}$ satisfies $\mu_{H}(\mathscr{F})(\leq)<\mu_{H}(\mathscr{E}/\mathscr{F})$.
      }
    \end{enumerate}
  \end{definition}

  Our definition of $\mu_{H}$-stability differs slightly from the standard definition (e.g. \cite{huybrechts2010}*{Definition 1.2.12})
  For torsion sheaves, our definition is significantly weaker: if $\operatorname{codim}(\mathscr{E})\geq 1$ then $\mathscr{E}$ is $\mu_{H}$-semistable.
  However, if $\mathscr{E}$ is torsion-free then Definition \ref{def:twistedSlopeStable} agrees with the standard definition.
  
  Any coherent sheaf has a filtration where the consecutive quotients are $\mu_{H}$-semistable with descending slopes.
  
  \begin{definition}
    \label{definition:harderNarasimhan}
    Suppose $\mathscr{E}\in\operatorname{Coh}(X)$ is nonzero.
    There exists a filtration, called a \textit{Harder-Narasimhan filtration}, of coherent sheaves
    \[
      0=\mathscr{E}_{0}\subseteq \mathscr{E}_{1}\subseteq\cdots\subseteq\mathscr{E}_{m-1}\subseteq\mathscr{E}_{m}=\mathscr{E}
    \]
    such that $\mathscr{E}_{i}/\mathscr{E}_{i-1}$ is $\mu_{H}$-semistable for $i=1,2,\ldots,m$ and
    \[
      \mu_{H}(\mathscr{E}_{1}/\mathscr{E}_{0})>\mu_{H}(\mathscr{E}_{2}/\mathscr{E}_{1})>\cdots>\mu_{H}(\mathscr{E}_{m}/\mathscr{E}_{m-1}).
    \]
    
    In this case, $\mathscr{E}_{1}$ is called a maximal destabilizing subsheaf and $\mathscr{E}/\mathscr{E}_{m-1}$ is called a minimal destabilizing quotient.
  \end{definition}
  
  A Harder - Narasimhan filtration is unique up to codimension $2$.
  In other words, any two Harder-Narasimhan filtration of $\mathscr{E}$ have the length, and the rank and degree of the corresponding consecutive quotients are equal.
  For this reason, the slope of a minimal or maximal consecutive quotient is independent of filtration and so well-defined.
  
  \begin{definition}
    \label{definition:slopeMaximalDestabilizing}
    Suppose $\mathscr{E}\in\operatorname{Coh}(X)$ is nonzero.
    We define $\mu_{H}^{+}(\mathscr{E})=\mu_{H}(\mathscr{F})$ where $0\to\mathscr{F}\to\mathscr{E}$ is a maximal destabilizing subsheaf.
    Similarly, we define $\mu_{H}^{-}(\mathscr{E})=\mu_{H}(\mathscr{G})$ where $\mathscr{E}\to\mathscr{G}\to0$ is a minimal destabilizing quotient.
  \end{definition}
  
  The quantity $\mu_{H}^{+}$ is well-behaved under extensions.
  The following lemma is proven by Butler for curves, but the same proof holds for surfaces.
  
  \begin{lemma}[\cite{butler1994}*{Lemma 2.4}]
    \label{lemma:boundMaximalInSes}
    If $0\to\mathscr{F}\to\mathscr{E}\to\mathscr{G}\to 0$ is a short exact sequence in $\operatorname{Coh}(X)$ then
    \[
      \mu_{H}^{+}(\mathscr{E})\leq\max\{\mu_{H}^{+}(\mathscr{F}),\mu_{H}^{+}(\mathscr{G})\}.
    \]
  \end{lemma}
  
  We end this subsection by discussing Bogomolov's inequality which says the discriminant of a $\mu_{H}$-semistable sheaf is nonzero.
  When discussing the wall and chamber structure for tilt stability a variant of the discriminant, which we call the $H$-discriminant, appears naturally.
  For this reason we only consider the $H$-discriminant.

  \begin{definition}
    \label{definition:hDDiscriminant}
    Suppose $X$ is a smooth projective surface and $H$ is an ample divisor on $X$
    We define the \textit{$H$-discriminant} of $\mathscr{E}$ as the integer
    \[
      \overline{\Delta}_{H}(\mathscr{E})=\operatorname{deg}_{H}(\mathscr{E})^{2}-2\operatorname{ch}_{2}(\mathscr{E})\operatorname{rank}(\mathscr{E}).
    \]
  \end{definition}
  
  We use the notation $\overline{\Delta}_{H}$ to denote both the dependence on the ample divisor $H$ and to distinguish the $H$-discriminant from the usual discriminant ($\Delta=\operatorname{ch}_{1}^{2}-2\operatorname{ch}_{2}\operatorname{rank}$).
  Either way, the Hodge index theorem~\cite{lazarsfeld2004}*{Corollary 1.6.3} and the usual Bogomolov inequality~\cite{huybrechts2010}*{Theorem 7.3.1} give the following bound on the $H$-discriminant for semistable sheaves.
  
  \begin{lemma}[Bogomolov's Inequality]
    \label{lemma:bogomolovInequality}
    Suppose $X$ is a smooth projective surface over an algebraic closed field of characteristic $0$.
    If $\mathscr{E}\in\operatorname{Coh}(X)$ is $\mu_{H}$-semistable then $\overline{\Delta}_{H}(\mathscr{E})\geq 0$.
  \end{lemma}
  
  We remark on the characteristic $0$ assumption in the Lemma above.
  
  \begin{remark}
    \label{remark:characteristicP}
    If the base field is of characteristic $p>0$ then Bogomolov's inequality does not hold.
    However, Langer shows Lemma \ref{lemma:bogomolovInequality} holds for surfaces of Kodaira dimension $\leq 1$ that are not quasielliptic~\cite{langer2016}*{Theorem 7.1}.
    In particular, Lemma \ref{lemma:bogomolovInequality} holds for Del Pezzo surfaces over any algebraically closed base field.
  \end{remark}

  \subsection*{Twisted Gieseker Stability}
  We now recall $(H,D)$-Gieseker stability.
  For fixed ample divisor $H$ and $\mathbb{R}$-divisor $D$, $(H,D)$-Gieseker stability is a variant of Gieseker stability due to Matsuki and Wentworth \cite{matsuki1997}*{Definition 3.2}.
  We introduce $(H,D)$-Gieseker stability because it arises naturally in the study of Bridgeland stability.
  However, $(H,D)$-Gieseker stability originally arose as a method to understand Gieseker stability under twists.
  In contrast to $\mu_{H}$-stability, Gieseker stability is \textit{not} invariant under twists by a line bundle and $(H,D)$-Gieseker stability partially accounts for this failure.
  Namely, for an integral divisor $D$, $\mathscr{E}(D)$ is Gieseker (semi)stable with respect to $H$ if and only if $\mathscr{E}$ is $(H,D)$-Gieseker (semi)stable.
  If $\operatorname{Pic}(X)=\mathbb{Z}$ and $D$ is a $\mathbb{R}$-divisor then $(H,D)$-Gieseker stability is equivalent to usual Gieseker stability.
  
  \begin{definition}
    \label{definition:giesekerStability}
    Assume $\mathscr{E}\in\operatorname{Coh}(X)$ and $D$ is a $\mathbb{R}$-divisor.
    \begin{enumerate}
      \item{
        We define the \textit{reduced $D$-twisted Hilbert polynomial} of $\mathscr{E}$ to be
        \[
          G_{H}^{D}(\mathscr{E})(t)=
          \begin{cases}
            \frac{\chi(\mathscr{E}\otimes\mathscr{O}_{X}(tH+D))}{\operatorname{rank}(\mathscr{E})}: & \operatorname{rank}(\mathscr{E})\neq 0 \\
            +\infty: & \operatorname{rank}(\mathscr{E})= 0.
          \end{cases}
        \]
        
        Note $\mathscr{O}_{X}(tH+D)$ is an abuse of notation.
        Namely, $D$ is a $\mathbb{R}$-divisor so $\mathscr{O}_{X}(tH+D)$ is not necessarily a line bundle.
        However, we can still define $\chi(\mathscr{E}\otimes\mathscr{O}_{X}(tH+D))$ formally as a polynomial in $t$ via the Hirzebruch-Riemann-Roch theorem.
      }
      \item{
        We say nonzero $\mathscr{E}$ is $(H,D)$-\textit{Gieseker (semi)stable} if every nonzero proper subsheaf $\mathscr{F}\subseteq\mathscr{E}$ satisfies $G_{H}^{D}(\mathscr{F})(t)(\leq)< G_{H}^{D}(\mathscr{E}/\mathscr{F})(t)$ for all $t\gg 0$.
      }
    \end{enumerate}
  \end{definition}

   The first few terms of the reduced $D$-twisted Hilbert polynomial can be written solely in terms of $\mu_{H}$, $D$, and invariants of $X$.
  This allows the following comparison of $\mu_{H}$-stability and $(H,D)$-Gieseker stability.
  \begin{lemma}
    \label{lemma:implicationsAmongCoherentStability}
    Fix an ample divisor $H$ and an $\mathbb{R}$-divisor $D$.
    Suppose $\mathscr{E}\in\operatorname{Coh}(X)$ is torsion-free.
    Each of the following statements implies the next.
    \begin{itemize}
      \item{
        $\mathscr{E}$ is $\mu_{H}$-stable
      }
      \item{
        $\mathscr{E}$ is $(H,D)$-Gieseker stable
      }
      \item{
        $\mathscr{E}$ is $(H,D)$-Gieseker semistable.
      }
      \item{
        $\mathscr{E}$ is $\mu_{H}$-semistable.
      }
    \end{itemize}
  \end{lemma}

\section{Tilt Stability}
  \label{section:tiltStability}
  In this section we recall the definition and relevant properties of tilt stability.
  Tilt stability is a family of Bridgeland stability conditions constructed by Bridgeland and Arcara-Betram~ \cites{bridgeland2008,arcara2013}.
  This family is parameterized by $\mathbb{R}\times\mathbb{R}_{>0}$ and, once we fix an object $E$, there is convenient wall and chamber structure where stability of $E$ can only change when crossing a wall.
  Moreover, the Large Volume Limit allows us to compare $\mu_{H}$-stability and $(H,\frac{K_{X}}{2})$-Gieseker stability to tilt stability.
  The results and definitions of this section are mostly well-known.
  The author has not seen Lemma \ref{lemma:maximalDestabilizingSubobjectIsInHeart} or Corollary \ref{lemma:boundOnLargestWall} explicitly written but expects they are well known to experts.
  
\subsection*{Preliminaries}
  For fixed ample divisor, tilt stability is a family of stability conditions on $X$ parameterized by $\mathbb{R}\times\mathbb{R}_{>0}$.
  For each $(\beta,\alpha)\in\mathbb{R}\times\mathbb{R}_{>0}$ there is a slope function $\mu_{\alpha,\beta}^{\mathrm{tilt}}:\operatorname{Coh}_{H}^{\beta}(X)\to\mathbb{R}\cup\{+\infty\}$ satisfying a positivity property.
  Furthermore, the function $\mu_{\alpha,\beta}^{\mathrm{tilt}}$ satisfies a variant of the Harder-Narasimhan filtration and Bogomolov's Inequality (c.f. Definition \ref{definition:harderNarasimhan} and Lemma \ref{lemma:bogomolovInequality})~\cites{bridgeland2008,arcara2013}. 
  
  \begin{definition}[\cite{bridgeland2008}*{Section 6}]
    \label{definition:explicitDescriptionOfHeart}
    Fix an integral ample divisor $H$ on $X$.
    For each $\beta\in\mathbb{R}$ we define
    \begin{itemize}
      \item{
        $\mathcal{T}_{H}^{\beta}(X)$ to be the full subcategory of $\operatorname{Coh}(X)$ generated by $\{\mathscr{E}\in\operatorname{Coh}(X)\mid \mu_{H}^{-}(\mathscr{E})>\beta\}$ where $\mu_{H}^{-}(\mathscr{E})$ is the slope of a minimal destabilizing quotient (see Definition \ref{definition:slopeMaximalDestabilizing}).
      }
      \item{
        $\mathcal{F}_{H}^{\beta}(X)$ to be the full subcategory of $\operatorname{Coh}(X)$ generated by $\{\mathscr{E}\in\operatorname{Coh}(X)\mid\mu_{H}^{+}(\mathscr{E})\leq\beta\}$ where $\mu_{H}^{+}(\mathscr{E})$ is the slope of a maximal destabilizing subobject.
      }
      \item{
        $\operatorname{Coh}_{H}^{\beta}(X)$ to be the full subcategory of $D^{b}(X)$ generated by
        \[
          \{E\in D^{b}(X)\mid \mathscr{H}^{-1}(E)\in\mathcal{F}_{H}^{\beta}(X),\mathscr{H}^{0}(E)\in\mathcal{T}_{H}^{\beta}(X),\mathscr{H}^{i}(E)=0 \ \text{if} \ i\neq -1,0\}
        \]
      }
    \end{itemize}
  \end{definition}
  
  By \cite{happel1996}*{Proposition 2.1} and \cite{bridgeland2008}*{Lemma 6.1}, for all $\beta\in\mathbb{R}$ the category $\operatorname{Coh}_{H}^{\beta}(X)$ is the heart of a bounded $t$-structure on $D^{b}(X)$.
  In particular, $\operatorname{Coh}_{H}^{\beta}(X)$ is an abelian category.
  
  Up to shifting, every coherent sheaf is an object in $\operatorname{Coh}_{H}^{\beta}(X)$.
  To explain, if we formally define $\mu_{H}^{+}(0)=-\infty$ and $\mu_{H}^{-}(0)=+\infty$ then $E\in\operatorname{Coh}_{H}^{\beta}(X)$ if and only if both $\beta\in [\mu_{H}^{+}(\mathscr{H}^{-1}(E)),\mu_{H}^{-}(\mathscr{H}^{0}(E)))$ and $\mathscr{H}^{i}(E)=0$ for $i\neq 0,-1$.
  In particular, $\mathscr{E}[1]\in\operatorname{Coh}_{H}^{\beta}(X)$ for all $\beta\geq\mu_{H}^{+}(\mathscr{E})$ and $\mathscr{E}\in\operatorname{Coh}_{H}^{\beta}(X)$ for all $\mu_{H}^{-}(\mathscr{E})<\beta$.
  
  On the other hand, an injection of coherent sheaves does \textit{not} induce an injection in $\operatorname{Coh}_{H}^{\beta}(X)$.
  However, the following lemma shows the inclusion of a maximal destabilizing subsheaf \textit{does} induce an injection.
  
  \begin{lemma}
    \label{lemma:maximalDestabilizingSubobjectIsInHeart}
    Suppose $\mathscr{E}\in\operatorname{Coh}(X)$ is nonzero.
    \begin{enumerate}
      \item{
        Assume $\mathscr{E}$ is not $\mu_{H}$-semistable.
        If $0\to\mathscr{F}\to\mathscr{E}$ is a maximal $\mu_{H}$-destabilizing subobject then $\mu_{H}^{+}(\mathscr{E}/\mathscr{F})\leq\mu_{H}(\mathscr{F})$.
        In particular, the morphism $\mathscr{F}[1]\to\mathscr{E}[1]$ in $D^{b}(X)$ induced by the inclusion is injective in $\operatorname{Coh}_{H}^{\beta}(X)$ for all $\beta\geq\mu_{H}^{+}(\mathscr{E})$.
      }
      \item{
        Assume $\mathscr{E}$ is $\mu_{H}$-semistable.
        If $0\to\mathscr{F}\to\mathscr{E}$ is a coherent subsheaf satisfying $\mu_{H}(\mathscr{F})=\mu_{H}(\mathscr{E})=\mu_{H}(\mathscr{E}/\mathscr{F})$ then the morphism $\mathscr{F}[1]\to\mathscr{E}[1]$ in $D^{b}(X)$ induced by the inclusion is injective in $\operatorname{Coh}_{H}^{\beta}(X)$ for all $\beta\geq\mu_{H}(\mathscr{E})$.
      }
    \end{enumerate}
  \end{lemma}
  \begin{proof}
    By \cite{beilinson1981}*{Page 32}, an injection $0\to\mathscr{F}[1]\to\mathscr{E}[1]$ in $\operatorname{Coh}_{H}^{\beta}(X)$ is equivalent to the inequality $\mathscr{F}[1],\mathscr{E}/\mathscr{F}[1]\in\operatorname{Coh}_{H}^{\beta}(X)$ for all $\beta\geq\mu_{H}^{+}(\mathscr{E})$.
    In other words,, it suffices to show $\beta\geq\mu_{H}^{+}(\mathscr{E})\geq\mu_{H}^{+}(\mathscr{E}/\mathscr{F}),\mu_{H}^{+}(\mathscr{F})$.
    By definition, $\mu_{H}^{+}(\mathscr{E})=\mu_{H}(\mathscr{F})=\mu_{H}^{+}(\mathscr{F})$, so it remains to show $\mu_{H}^{+}(\mathscr{E}/\mathscr{F})\leq\mu_{H}(\mathscr{F})$.
    In part (2) this is true by assumption, so we only consider part (1). 
    
    We proceed by induction on the length of a Harder-Narasimhan filtration of $\mathscr{E}$.
    If $\mathscr{E}$ has a Harder-Narasimhan filtration of length $2$ then $\mathscr{E}/\mathscr{F}$ is $\mu_{H}$-stable with $\mu_{H}(\mathscr{E}/\mathscr{F})<\mu_{H}(\mathscr{F})$, as claimed.
    Thus, consider a Harder-Narasimhan filtration
    \[
      0=\mathscr{E}_{0}\to\mathscr{E}_{1}\to\cdots\to\mathscr{E}_{m-1}\to\mathscr{E}_{m}=\mathscr{E}.
    \]
    By the inductive hypothesis, we know $\mu_{H}^{+}(\mathscr{E}_{m-1}/\mathscr{E}_{1})<\mu_{H}(\mathscr{E}_{1})$.
    We have the following short exact sequence
    \[
      0\to\mathscr{E}_{m-1}/\mathscr{E}_{1}\to\mathscr{E}/\mathscr{E}_{1}\to\mathscr{E}/\mathscr{E}_{m-1}\to 0.
    \]
    Therefore, by Lemma \ref{lemma:boundMaximalInSes},
    \[
      \mu_{H}^{+}(\mathscr{E}/\mathscr{E}_{1})\leq\max\{\mu_{H}^{+}(\mathscr{E}_{m-1}/\mathscr{E}_{1}),\mu_{H}^{+}(\mathscr{E}/\mathscr{E}_{m-1})\}.
    \]
    However, by definition of a Harder-Narasimhan filtration $\mu_{H}^{+}(\mathscr{E}/\mathscr{E}_{m-1})=\mu_{H}(\mathscr{E}/\mathscr{E}_{m-1})<\mu_{H}(\mathscr{E}_{1})$, and by the inductive hypothesis $\mu_{H}^{+}(\mathscr{E}_{m-1}/\mathscr{E}_{1})\leq\mu_{H}(\mathscr{E}_{1})$.
    Hence, we have obtained the desired inequality.
  \end{proof}
  
  With $\operatorname{Coh}_{H}^{\beta}(X)$ defined, we can define the slope function for tilt stability.
  
  \begin{definition}
    Fix an ample divisor $H$ on $X$ and $(\beta,\alpha)\in\mathbb{R}\times\mathbb{R}_{>0}$.
    We say $E\in\operatorname{Coh}_{H}^{\beta}(X)$ is \textit{$\mu_{\alpha,\beta}^{\mathrm{tilt}}$-(semi)stable} with respect to $(\beta,\alpha)$ if every subobject $0\to F\to E$ in $\operatorname{Coh}_{H}^{\beta}(X)$ satisfies $\mu_{\alpha,\beta}^{\mathrm{tilt}}(F)(\leq)<\mu_{\alpha,\beta}^{\mathrm{tilt}}(E/F)$ where
    \[
      \mu_{\alpha,\beta}^{\mathrm{tilt}}(E)=
      \begin{cases}
        \frac{\operatorname{ch}_{2}(E)-\beta\operatorname{deg}_{H}(E)+\left(\frac{\beta^{2}}{2}-\frac{\alpha^{2}}{2}\right)\operatorname{rank}(E)}{\operatorname{deg}_{H}(E)-\beta\operatorname{rank}(E)}: & \operatorname{deg}_{H}(E)- \beta\operatorname{rank}(E)\neq 0 \\
        +\infty: & \operatorname{deg}_{H}(E)-\beta\operatorname{rank}(E)=0.
      \end{cases}
    \]
  \end{definition}
  
  In fact, $\mu_{\alpha,\beta}^{\mathrm{tilt}}$ is known to be a \textit{Bridgeland stability condition}.
  Furthermore, using a variant of the Harder-Narasimhan filtration one can define $\mu_{\alpha,\beta}^{\mathrm{tilt}\pm}$ similarly to Definition \ref{definition:harderNarasimhan}.
  The functions $\mathbb{R}\times\mathbb{R}_{>0}\to\mathbb{R}\cup\{+\infty\}$ given by $(\beta,\alpha)\mapsto \mu_{\alpha,\beta}^{\mathrm{tilt}+}(E)$ and $\mu_{\alpha,\beta}^{\mathrm{tilt}-}(E)$ are continuous for all $E\in D^{b}(X)$ (where $\mathbb{R}\times\mathbb{R}_{>0}$ has the usual topology).
  For this reason, by letting $(\beta,\alpha)\in\mathbb{R}\times\mathbb{R}_{>0}$ vary, we view $\mathbb{R}\times\mathbb{R}_{>0}$ as a parameter space of tilt stability.

  \begin{definition}
    \label{definition:hDSlice}
    The family $\mathbb{R}\times\mathbb{R}_{>0}$ parameterizing $\mu_{\alpha,\beta}^{\mathrm{tilt}}$-stability is called the \textit{$(\alpha,\beta)$-plane}.
  \end{definition}
  
  The $(\alpha,\beta)$-plane is also called the $H$-slice.
  Also, the notation of the $(\alpha,\beta)$-plane is a little misleading: following convention $\alpha$ denotes the vertical axis while $\beta$ denotes the horizontal axis.

  We end this subsection by recalling Bridgeland's Large Volume Limit.
  This is one of the main tools for relating $\mu_{H}$-stable and $(H,-\frac{K_{X}}{2})$-Giesker stable sheaves to $\mu_{\alpha,\beta}^{\mathrm{tilt}}$-stable objects.
  
  \begin{lemma}[Large Volume Limit \cite{bridgeland2008}*{Proposition 14.2}]
    \label{lemma:largeVolumeLimit}
    Assume $\mathscr{E}\in\operatorname{Coh}(X)$ is torsion-free.
    \begin{enumerate}
      \item{
        Fix $\beta<\mu_{H}^{-}(\mathscr{E})$.
        $\mathscr{E}$ is $(H,-\frac{K_{X}}{2})$-Gieseker stable if and only if $\mathscr{E}$ is $\mu_{\alpha,\beta}^{\mathrm{tilt}}$-stable for all $\alpha\gg 0$.
      }
      \item{
        Fix $\beta>\mu_{H}^{+}(\mathscr{E})$
        If $\mathscr{E}[1]$ is $\mu_{\alpha,\beta}^{\mathrm{tilt}}$-stable for all $\alpha\gg 0$ then $\mathscr{E}$ is $\mu_{H}$-stable.
      }
      \item{
        Fix $\beta>\mu_{H}^{+}(\mathscr{E})$.
        If $\mathscr{E}$ is locally free and $\mu_{H}$-stable then $\mathscr{E}[1]$ is $\mu_{\alpha,\beta}^{\mathrm{tilt}}$-stable for all $\alpha\gg 0$.
      }
    \end{enumerate}
  \end{lemma}
  
\subsection*{Walls for Tilt Stability}
  In this subsection we discuss the wall and chamber structures of the $(\alpha,\beta)$-plane.
  Namely, for fixed object $E\in D^{b}(X)$, there is a well-behaved collection of real $1$ dimensional submanifolds, called walls, of the $(\alpha,\beta)$-plane where $\mu_{\alpha,\beta}^{\mathrm{tilt}}$-stability of $E$ can only change when crossing a wall.
  
  We begin by describing numerical walls.
  Numerical walls are a collection of real $1$ dimensional submanifolds of the $(\alpha,\beta)$-plane with purely numerical description.
  Importantly, numerical walls contain all actual walls so any restrictions on numerical walls also constrain the collection of actual walls.
  
  \begin{definition}
    A \textit{numerical wall} associated $E\in D^{b}(X)$ is the set of points in the $(\alpha,\beta)$-plane of the form:
    \[
      W(E,F)=\{(\beta,\alpha)\in\mathbb{R}\times\mathbb{R}_{>0}\mid\mu_{\alpha,\beta}^{\mathrm{tilt}}(E)=\mu_{\alpha,\beta}^{\mathrm{tilt}}(F)\}
    \]
    for some object $F\in D^{b}(X)$.
  \end{definition}
  
  Numerical walls in the $(\alpha,\beta)$-plane are constrained. 
  The following result states that in reasonable cases numerical walls consists of a unique vertical line and nested semicircles on each side of that vertical line.

  \begin{lemma}[\cite{maciocia2014}]
    \label{lemma:wallsOfTiltStability}
    Fix an object $E\in D^{b}(X)$.
    \begin{enumerate}
      \item{
        $W(E,F)$ is given by the equation:
        \[
          x\alpha^2+x\beta^2-2y\beta+2z=0
        \]
        where
        \begin{align*}
          x &=\operatorname{deg}_{H}(E)\operatorname{rank}(F)-\operatorname{deg}_{H}(F)\operatorname{rank}(E)\\
          y &=\operatorname{ch}_{2}(E)\operatorname{rank}(F)-\operatorname{ch}_{2}(F)\operatorname{rank}(E)\\
          z &=\operatorname{ch}_{2}(E)\operatorname{deg}_{H}(F)-\operatorname{ch}_{2}(F)\operatorname{deg}_{H}(E)
        \end{align*}
        
        In particular,
        \begin{itemize}
          \item{
            if $x\neq 0$ then $W(E,F)$ is a semicircle centered at $(y/x,0)$ with radius squared $y^2/x^2-2z/x$
          }
          \item{
            if $x=0$ and $y\neq 0$ then $W(E,F)$ is the vertical line described by $\beta=-z/y$,
          }
          \item{
            if $x,y=0$ then $W(E,F)=\varnothing$ or $W(E,F)=\mathbb{R}\times\mathbb{R}_{>0}$
          }
        \end{itemize}
      }
      \item{
        If $\operatorname{rank}(E)\neq 0$ and $\mu_{H}(E)\neq\mu_{H}(F)$ then $W(E,F)$ is a semicircle with radius squared
        \[
          \rho^{2}=(\mu_{H}(E)-c)^2-\frac{\overline{\Delta}_{H}(E)}{\operatorname{rank}(E)^{2}}
        \]
        where $(c,0)$ is the center of the semicircle.
        Recall $\overline{\Delta}_{H}(E)$ is the $H$-discriminant of Definition \ref{definition:hDDiscriminant}.
      }
      \item{
        If $\operatorname{rank}(E)\neq 0$ then there is a unique vertical numerical wall given by the equation $\beta=\mu_{H}(E)$.
      }
      \item{
        If $\overline{\Delta}_{H}(E)\geq 0$ then any two numerical walls are disjoint in the $(\alpha,\beta)$-plane.
      }
    \end{enumerate}
  \end{lemma}
  
  If $\overline{\Delta}_{H}(E)<0$ then numerical walls are not disjoint.
  To explain, if $\operatorname{rank}(E)\neq 0$ and $\overline{\Delta}_{H}(E)<0$ then every semicircular wall intersects the line $\beta=\mu_{H}(E)$.
  Furthermore, numerical walls may intersect along the line $\alpha=0$ in the $(\alpha,\beta)$-plane.
  For example, any numerical wall in the $(\alpha,\beta)$-plane associated to $\mathscr{O}_{X}$ intersects the point $(0,0)$.
  However, if $\overline{\Delta}_{H}(E)>0$ then numerical walls will not even intersect along the line $\alpha=0$.
  
  Actual walls are numerical walls associated to $E$ where $\mu_{\alpha,\beta}^{\mathrm{tilt}}$-stability may change.
  In other words, $\mu_{\alpha,\beta}^{\mathrm{tilt}}$-stability (or unstability) is independent of $(\beta,\alpha)$ within chambers. 
  
  \begin{definition}
    We say that a numerical wall $W$ is an \textit{actual wall} associated to $E$ if for some $(\beta,\alpha)\in W$ there exists a short exact sequence
    \[
      0\to F\to E\to E/F\to 0
    \]
    in $\operatorname{Coh}_{H}^{\beta }(X)$ such that 
    \begin{itemize}
      \item{
        $W=W(E,F)$,
      }
      \item{
        $F$, $E$, and $E/F$ are $\mu_{\alpha,\beta}^{\mathrm{tilt}}$-semistable (in particular, $F,E,E/F\in\operatorname{Coh}_{H}^{\beta}(X)$),
      }
      \item{
        $\mu_{\alpha,\beta}^{\mathrm{tilt}}(F)=\mu_{\alpha,\beta}^{\mathrm{tilt}}(E)=\mu_{\alpha,\beta}^{\mathrm{tilt}}(E/F)$, and
      }
      \item{
        $W\neq\mathbb{R}\times\mathbb{R}_{>0}$.
      }
    \end{itemize}
    
    We call $0\to F\to E\to E/F\to 0$ a destabilizing sequence associated to $W$ or a destabilizing sequence associated to $E$.
  \end{definition}
  
  In fact, a destabilizing sequence associated to $\mathscr{E}$ or $\mathscr{E}[1]$ is independent of $(\beta,\alpha)\in W$.
  A proof of this fact in the case of $\mathbb{P}^{2}$ and $\operatorname{ch}_{2}(E)<0$ appears in \cite{arcaraB2013}*{Lemma 6.3}, but the same argument holds in the general case~\cite{kopper2020}*{Lemma 3.1}.
  
  Bridgeland showed for fixed $E\in D^{b}(X)$ that the collection of actual walls is locally finite and stability can only change when crossing a wall.
  
  \begin{lemma}[\cite{bridgeland2008}*{Section 9}]
    \label{lemma:stabilityIsConstantInChambers}
    Assume $E\in D^{b}(X)$ is nonzero.
    Let $\mathcal{W}$ be the collection of all actual walls associated to $E$ in the $(\alpha,\beta$)-slice.
    \begin{enumerate}
      \item{
        The collection $\mathcal{W}$ is locally finite.
        In other words, for any compact subset $V\subseteq\mathbb{R}\times\mathbb{R}_{>0}$, $V$ intersects finitely many walls in $\mathcal{W}$.
      }
      \item{
        If $E$ is $\mu_{\alpha_{0},\beta_{0}}^{\mathrm{tilt}}$-stable for some $(\beta_{0},\alpha_{0})$ in $(\mathbb{R}\times\mathbb{R}_{>0})\setminus(\bigcup_{W\in\mathcal{W}} W)$ then $E$ is $\mu_{\alpha_{0},\beta_{0}}^{\mathrm{tilt}}$-stable for all $(\beta,\alpha)$ in the connected component in $(\mathbb{R}\times\mathbb{R}_{>0})\setminus(\bigcup_{W\in\mathcal{W}} W)$ containing $(\beta_{0},\alpha_{0})$.
        
        In other words, tilt stability is constant in chambers.
      }
    \end{enumerate}
  \end{lemma}
  
  Using \cite{coskun2016}*{Proposition 8.3} and a variant of \cite{kopper2020}*{Theorem 3.3} the largest actual wall associated to $\mathscr{E}$ or $\mathscr{E}[1]$ in $D^{b}(X)$ can be effectively bounded:
  \begin{lemma}
    \label{lemma:boundOnLargestWall}
    Let $\mathscr{E}$ be a torsion-free sheaf with $\overline{\Delta}_{H}(\mathscr{E})\geq 0$.
    Assume $W$ is an actual semicircular wall associated to $\mathscr{E}$ (resp. $\mathscr{E}[1]$) with center $(c,0)$, radius $\rho$, and destabilizing sequence $0\to F\to\mathscr{E}\to G\to 0$ (resp. $0\to F\to\mathscr{E}[1]\to G\to 0$).
    \begin{enumerate}
      \item{
        If $\mathscr{H}^{-1}(G)\neq 0$ (resp. $\operatorname{rank}(\mathscr{H}^{0}(F))\neq 0$) then
        \[
          \rho\leq\sqrt{\frac{\overline{\Delta}_{H}(\mathscr{E})}{4(\operatorname{rank}(\mathscr{E})+1)}}
        \]
        where $W$ is empty if $\overline{\Delta}_{H}(\mathscr{E})=0$.
      }
      \item{
        If $\mathscr{H}^{-1}(G)=0$ (resp. $\operatorname{rank}(\mathscr{H}^{0}(F))=0$) then
        \[
          \rho\leq\frac{1}{2}\left|\overline{\Delta}_{H}(\mathscr{E})-\frac{1}{\operatorname{rank}(\mathscr{E})^{2}}\right|.
        \]
        where $W$ is empty if $\overline{\Delta}_{H}(\mathscr{E})=0$ or $\overline{\Delta}_{H}(\mathscr{E})=1$ with $\operatorname{rank}(\mathscr{E})=1$.
      }
    \end{enumerate}
  \end{lemma}
  \begin{proof}
    We only consider a semicircular wall associated to $\mathscr{E}$: the argument for $\mathscr{E}[1]$ is essentially the same.
    Furthermore, if $\overline{\Delta}_{H}(\mathscr{E})=0$ then $W$ is empty \cite{bayer2016}*{Proposition A.8}, so we may assume $\overline{\Delta}_{H}(\mathscr{E})>0$.
    Taking cohomology of the destabilizing sequence gives the following exact sequence in $\operatorname{Coh}(X)$:
    \[
      0\to\mathscr{H}^{-1}(G)\to\mathscr{H}^{0}(F)\to\mathscr{E}\to\mathscr{H}^{0}(G)\to 0.
    \]
    \begin{enumerate}
      \item{
        If $\mathscr{H}^{-1}(G)\neq 0$, by \cite{coskun2016}*{Proposition 8.3},
        \[
          \rho^{2}\leq \frac{\overline{\Delta}_{H}(\mathscr{E})}{4(\operatorname{rank}(\mathscr{E})+1)},
        \]
        as claimed.
      }
      \item{
        Assume $\mathscr{H}^{-1}(G)=0$.
        In this case, we give slight simplification of \cite{kopper2020}*{Theorem 3.3}.
        Since $\mathscr{H}^{-1}(G)=0$, we have the following exact sequence in $\operatorname{Coh}(X)$:
        \[
          0\to\mathscr{H}^{0}(F)\to\mathscr{E}\to\mathscr{H}^{0}(G)\to 0.
        \]
        By the seesaw inequality~\cite{rudakov1997}*{Lemma 3.2}, one of the following inequalities must hold:
        \begin{itemize}
          \item{
            $\mu_{H}(\mathscr{H}^{0}(F))\leq\mu_{H}(\mathscr{E})$
          }
          \item{
            $\mu_{H}(\mathscr{H}^{0}(G))\leq\mu_{H}(\mathscr{E})$.
          }
        \end{itemize}
        
        First suppose $\mu_{H}(\mathscr{H}^{0}(F))\leq \mu_{H}(\mathscr{E})$.
        Since $W$ is a semicircular wall, by (1) of Lemma \ref{lemma:wallsOfTiltStability}, $\mu_{H}(\mathscr{H}^{0}(F))\neq\mu_{H}(\mathscr{E})$.
        In other words, $\mu_{H}(\mathscr{H}^{0}(F))< \mu_{H}(\mathscr{E})$ so
        \[
          \mu_{H}(\mathscr{H}^{0}(F))+\frac{1}{\operatorname{rank}(\mathscr{E})^{2}}\leq\mu_{H}(\mathscr{E}).
        \]
        Therefore, by \cite{kopper2020}*{Lemma 3.1},
        \[
          c+\rho\leq\mu_{H}(\mathscr{H}^{0}(F))\leq\mu_{H}(\mathscr{E})-\frac{1}{\operatorname{rank}(\mathscr{E})^{2}}
        \]
        We obtain the same inequality in the case of $\mu_{H}(\mathscr{H}^{0}(G))\leq\mu_{H}(\mathscr{E})$.
        
        In other words, $W$ is contained in the semicircle with right endpoint $\mu_{H}(\mathscr{E})-\frac{1}{\operatorname{rank}(\mathscr{E})^{2}}$ whose center and radius satisfy the (2) of Lemma \ref{lemma:wallsOfTiltStability}.
         In other words,
         \[
           \rho\leq \frac{1}{2}\left|\overline{\Delta}_{H}(\mathscr{E})-\frac{1}{\operatorname{rank}(\mathscr{E})^{2}}\right|,
         \]
         as claimed.
      }
    \end{enumerate}
  \end{proof}
  
  In practice, we do not distinguish the cases in Lemma \ref{lemma:boundOnLargestWall}.
  For this reason, we introduce the following immediate corollary.
  \begin{corollary}
    \label{corollary:boundIndependentOfCases}
    Suppose $\mathscr{E}$ is a torsion-free sheaf with $\overline{\Delta}_{H}(\mathscr{E})$.
    If $W$ is an actual semicircular wall associated to $\mathscr{E}$ or $\mathscr{E}[1]$ with radius $\rho$ then
    \[
      \rho\leq\max\left\{\sqrt{\frac{\overline{\Delta}_{H}(\mathscr{E})}{4(\operatorname{rank}(\mathscr{E})+1)}},\frac{1}{2}\left|\overline{\Delta}_{H}(\mathscr{E})-\frac{1}{\operatorname{rank}(\mathscr{E})^{2}}\right|\right\}.
    \]
    Moreover, if $\overline{\Delta}_{H}(\mathscr{E})=0$ then there are no actual semicircular walls associated to $\mathscr{E}$ or $\mathscr{E}[1]$.
  \end{corollary}
  
  The results discussed above are collected in Figure \ref{figure:actualWalls}.
  
  \begin{figure}
    \centering
    \begin{tikzpicture}
      \draw[fill=gray!20,draw=gray!20] (0,0) -- (0.535898,0) arc (180:0:3.4641) -- (13,0) -- (13,5.5) -- (0,5.5)-- (0,0);
      \draw[<->,dotted] (-1,0) -- (13.5,0) node[right] {$\beta$};
      \draw[->,dotted] (-.5,0) -- (-.5,4.5) node[above] {$\alpha$};
      \draw[very thick] (0,0) -- (0,5.5) node[above] {$\beta=\mu_{H}(\mathscr{E})$};
      \draw[very thick] (0.535898,0) arc (180:0:3.4641);
      \draw[very thick] (0.763932,0) arc (180:0:2.23607);
      \draw[very thick] (1,0) arc (180:0:1.5);
      \draw[very thick] (1.45969,0) arc (180:0:0.640312);
      \draw[very thick] (1.80975,0) arc (180:0:0.20025);
      \draw[densely dashed] (-.5,3.75) -- (5.25,3.75) node[right] {$\alpha=\max\left\{\sqrt{\frac{\overline{\Delta}_{H}(\mathscr{E})}{4(\operatorname{rank}(\mathscr{E})+1)}},\frac{1}{2}\left|\overline{\Delta}_{H}(\mathscr{E})-\frac{1}{\operatorname{rank}(\mathscr{E})^{2}}\right|\right\}$};
      \node[text width=5cm] at (3.5,4.5) {$\mu_{\alpha,\beta}^{\mathrm{tilt}}$-stability of $\mathscr{E}[1]$ is equivalent to $\mu_{H}$-stability of $\mathscr{E}$};
    \end{tikzpicture}
    \caption{
      \label{figure:actualWalls}
      Actual walls associated to $\mathscr{E}[1]$ for $\mathscr{E}\in\operatorname{Coh}(X)$ torsion-free.
      The solid lines represent actual walls of $\mathscr{E}[1]$.
      The dashed line represents our bound on the largest actual wall.
      The shaded region is the large volume limit.
      The actual walls of $\mathscr{E}$ form a similar picture mirrored over $\beta=\mu_{H}(\mathscr{E})$ except $\mu_{\alpha,\beta}^{\mathrm{tilt}}$-stability of $\mathscr{E}$ is equivalent to $(H,-\frac{K_{X}}{2})$-Gieseker stability of $\mathscr{E}$.
    }
  \end{figure}
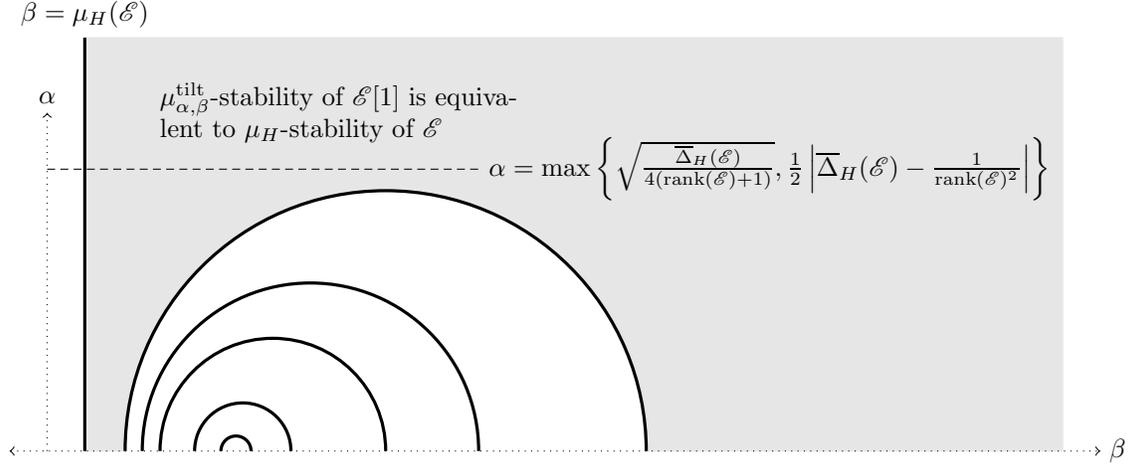
  
  We end this section by recalling a wall-crossing result for the $(\alpha,\beta)$-plane.
  Suppose $W$ is an actual semicircular wall associated to $E\in D^{b}(X)$.
  If $E$ is $\mu_{\alpha,\beta}^{\mathrm{tilt}}$-stable on one side of the semicircular wall then $E$ is $\mu_{\alpha,\beta}^{\mathrm{tilt}}$-unstable on the other side of the semicircular wall.
  The converse is false, $E$ may be $\mu_{\alpha,\beta}^{\mathrm{tilt}}$-unstable on both sides of an actual semicircular wall.
  Bayer and Macr{\`i}'s wall-crossing result gives conditions to guarantee the converse holds.
  In practice, this result is used to construct $\mu_{\alpha,\beta}^{\mathrm{tilt}}$-stable objects.
  Technically, we give the ``dual" version of their result, but the same argument holds.
  
  \begin{lemma}[\cite{bayer2011}*{Lemma 5.9}]
    \label{theorem:wallCrossingOne}
    Suppose $0\to F\to E\to G^{\oplus r}\to 0$ is exact in $\operatorname{Coh}_{H}^{\beta_{0}}(X)$ for some $\beta_{0}\in\mathbb{R}$.
    If there exists $\alpha_{0}>0$ satisfying
    \begin{itemize}
      \item{
        $F$ and $G$ are $\mu_{\alpha_{0},\beta_{0}}^{\mathrm{tilt}}$-stable,
      }
      \item{
        $\mu_{\alpha_{0},\beta_{0}}^{\mathrm{tilt}}(F)=\mu_{\alpha_{0},\beta_{0}}^{\mathrm{tilt}}(G)$,
      }
      \item{
        $\operatorname{Hom}_{D^{b}(X)}(G,E)=0$, and
      }
      \item{
        there exists $\varepsilon>0$ such that $\mu_{\alpha_{0},\beta_{0}}^{\mathrm{tilt}}(F)<\mu_{\alpha_{0},\beta_{0}}^{\mathrm{tilt}}(G)$ for all $\alpha\in (\alpha_{0},\alpha_{0}+\varepsilon)$
      }
    \end{itemize}
    then there exists $\overline{\Delta}_{H}>0$ such that $E$ is $\mu_{\alpha_{0},\beta_{0}}^{\mathrm{tilt}}$-stable for all $\alpha\in (\alpha_{0},\alpha_{0}+\overline{\Delta}_{H})$.
  \end{lemma}

\section{Stability Of Kernel Sheaves}
  \label{section:stabilityOfSyzygyBundles}
  In this section, we apply the theory above to show $\mu_{H}$-stability of kernel sheaves associated to sufficiently positive, $(-K_{X},-\frac{K_{X}}{2})$-Gieseker stable sheaves on Del Pezzo surfaces.
  
  \begin{definition}
    Assume $\mathscr{E}$ is a globally generated, torsion-free sheaf.
    In other words, there is a short exact sequence
    \begin{equation}
      \label{equation:sesForGloballyGenerated}
      0\to\mathscr{M}_{\mathscr{E}}\to H^{0}(\mathscr{E})\otimes\mathscr{O}_{X}\to\mathscr{E}\to 0.
    \end{equation}
    We call $\mathscr{M}_{\mathscr{E}}$ the \textit{kernel sheaf} associated to $\mathscr{E}$.
    If $\mathscr{E}$ is clear from context, we just write $\mathscr{M}$ instead of $\mathscr{M}_{\mathscr{E}}$.
  \end{definition}
  
  Kernel sheaves are also called syzygy sheaves, Lazarsfeld-Mukai sheaves, and Lazarsfeld sheaves.
  
  We provide an overview of this section.
  Our method for proving $\mu_{H}$-stability of $\mathscr{M}$ follows a wall-crossing method that has been used in many recent results~\cites{bayer2018,kopper2020,feyzbakhsh2022}.
  We describe this method in more detail for our scenario.
  By shifting Equation (\ref{equation:sesForGloballyGenerated}) we obtain a distinguished triangle:
  \[
    \mathscr{E}\to\mathscr{M}[1]\to (H^{0}(\mathscr{E})\otimes\mathscr{O}_{X})[1]\to \mathscr{E}[1]
  \]
  in $D^{b}(X)$.
  If $\mathscr{E}$ is $(-K_{X},\frac{K_{X}}{2})$-Giesker stable and sufficiently positive then $W((H^{0}(\mathscr{E})\otimes\mathscr{O}_{X})[1],\mathscr{E})$ is an actual wall associated to $\mathscr{M}[1]$.
  By applying Lemma \ref{theorem:wallCrossingOne}, $\mathscr{M}[1]$ is $\mu_{\alpha,\beta}^{\mathrm{tilt}}$-stable in the region just above the wall $W((H^{0}(\mathscr{E})\otimes\mathscr{O}_{X})[1],\mathscr{E})$.
  
  Once we have $\mu_{\alpha,\beta}^{\mathrm{tilt}}$-stability of $\mathscr{M}[1]$, the techniques for proving $\mu_{H}$-stability of $\mathscr{M}$ tends to diverge.
  For example \cite{bayer2018} is able to apply the Large Volume Limit (Lemma \ref{lemma:largeVolumeLimit}) directly.
  In contrast, \cite{kopper2020} uses the fact that for a torsion sheaf $\mu_{\alpha,\beta}^{\mathrm{tilt}}$-stability is equivalent to $\mu_{H}$-stability along its support (\cite{kopper2020}*{Lemma 2.6}).
  In comparison, \cite{feyzbakhsh2022}*{Theorem 1.3} fundamentally uses the geometry of $K3$ surfaces.
  
  In our scenario, we show $\mu_{\alpha,\beta}^{\mathrm{tilt}}$-stability of $\mathscr{M}[1]$ constrains the degree and second Chern character of a maximal $\mu_{H}$-destabilizing subsheaf of $\mathscr{M}$.
  In the case of Del Pezzo surfaces, we then use these constraints to show $\mathscr{M}$ is $\mu_{H}$-stable.
  
  \begin{lemma}
     \label{lemma:LazarsfeldStableAboveWall}
    Let $\mathscr{E}$ be a globally generated, torsion-free, $(H,-\frac{K_{X}}{2})$-Giesker stable sheaf on $X$ with associated kernel sheaf $\mathscr{M}$.
    If the following bounds are satisfied:
    \begin{itemize}
      \item{
        $0<\operatorname{deg}_{H}(\mathscr{E})$,
      }
      \item{
        $0<\operatorname{ch}_{2}(\mathscr{E})$, and
      }
      \item{
        $\overline{\Delta}_{H}(\mathscr{E})<2\frac{\operatorname{ch}_{2}(\mathscr{E})}{\operatorname{deg}_{H}(\mathscr{E})}+\frac{1}{\operatorname{rank}(\mathscr{E})^{2}}$
      }
    \end{itemize}
    then $W((H^{0}(\mathscr{E})\otimes\mathscr{O}_{X})[1],\mathscr{E})$ is an actual wall associated to $\mathscr{M}[1]$ in the $(\alpha,\beta)$-plane, and $\mathscr{M}[1]$ is $\mu_{\alpha,\beta}^{\mathrm{tilt}}$-stable in the chamber directly above this wall.    
  \end{lemma}
  \begin{proof}
    There is an exact sequence in $\operatorname{Coh}(X)$:
    \[
      0\to\mathscr{M}\to H^{0}(\mathscr{E})\otimes\mathscr{O}_{X}\to\mathscr{E}\to 0
    \]
    which induces the following exact sequence in $\operatorname{Coh}_{H}^{\beta H}(X)$ for all $\beta\in [0,\mu_{H}(\mathscr{E}))$:
    \[
      0\to\mathscr{E}\to\mathscr{M}[1]\to (H^{0}(\mathscr{E})\otimes\mathscr{O}_{X})[1]\to 0.
    \]
    Since $\mathscr{O}_{X}$ is $\mu_{H}$-stable and $\mathscr{E}$ is $(H,-\frac{K_{X}}{2})$-Gieseker stable, by the Large Volume Limit (Lemma \ref{lemma:largeVolumeLimit}), $\mathscr{O}_{X}[1]$ and $\mathscr{E}$ are $\mu_{\alpha,\beta}^{\mathrm{tilt}}$-stable for all $\beta\in [0,\mu_{H}(\mathscr{E}))$ and $\alpha\gg 0$.
    
    Moreover, since $\overline{\Delta}_{H}(\mathscr{O}_{X})=0$, by Corollary \ref{corollary:boundIndependentOfCases}, $\mathscr{O}_{X}[1]$ is $\mu_{\alpha,\beta}^{\mathrm{tilt}}$-stable for all $\beta\geq 0$ and all $\alpha>0$.
    Similarly, since $2\frac{\operatorname{ch}_{2}(\mathscr{E})}{\operatorname{deg}_{H}(\mathscr{E})}+\frac{1}{\operatorname{rank}(\mathscr{E})^{2}}>\overline{\Delta}_{H}(\mathscr{E})$,
    \[
      \max\left\{\sqrt{\frac{\overline{\Delta}_{H}(\mathscr{E})}{4(\operatorname{rank}(\mathscr{E})+1)}},\frac{1}{2}\left|\overline{\Delta}_{H}(\mathscr{E})-\frac{1}{\operatorname{rank}(\mathscr{E})^{2}}\right|\right\}< \frac{\operatorname{ch}_{2}(\mathscr{E})}{\operatorname{deg}_{H}(\mathscr{E})}.
    \]
    Therefore, by Corollary \ref{corollary:boundIndependentOfCases}, every actual wall of $\mathscr{E}$ must have radius smaller than $\frac{\operatorname{ch}_{2}(\mathscr{E})}{\operatorname{deg}_{H}(\mathscr{E})}$.
    In particular, by (1) of Lemma \ref{lemma:wallsOfTiltStability}, $W( (H^{0}(\mathscr{E})\otimes\mathscr{O}_{X})[1],\mathscr{E})$ is larger than any actual wall of $\mathscr{E}$.
    Therefore, $\mathscr{E}$ is $\mu_{\alpha,\beta}^{\mathrm{tilt}}$-stable for all $(\beta,\alpha)$ lying on or above the wall $W( (H^{0}(\mathscr{E})\otimes\mathscr{O}_{X})[1],\mathscr{E})$.
    In short, we have shown $W((H^{0}(\mathscr{E})\otimes\mathscr{O}_{X})[1],\mathscr{E})$ is an actual wall associated to $\mathscr{M}[1]$ in the $(\alpha,\beta)$-plane.
    
    It remains to show $\mathscr{M}[1]$ is $\mu_{\alpha,\beta}^{\mathrm{tilt}}$-stable for all $(\beta,\alpha)$ in the chamber directly above $W((H^{0}(\mathscr{E})\otimes\mathscr{O}_{X})[1],\mathscr{E})$.
    To this end, we appeal to Lemma \ref{theorem:wallCrossingOne}.
    First, we just showed $\mathscr{E}$ and $\mathscr{O}_{X}[1]$ are $\mu_{\alpha,\beta}^{\mathrm{tilt}}$-stable for all $(\beta,\alpha)\in W((H^{0}(\mathscr{E})\otimes\mathscr{O}_{X})[1],\mathscr{E})$.
    Second, by definition, $\mu_{\alpha,\beta}^{\mathrm{tilt}}((H^{0}(\mathscr{E})\otimes\mathscr{O}_{X})[1])=\mu_{\alpha,\beta}^{\mathrm{tilt}}(\mathscr{E})$ for all $(\beta,\alpha))\in W((H^{0}(\mathscr{E})\otimes\mathscr{O}_{X})[1],\mathscr{E})$.
    Third, by definition of a kernel sheaf,
    \[
      \operatorname{Hom}_{D^{b}(X)}((H^{0}(\mathscr{E})\otimes\mathscr{O}_{X})[1],\mathscr{M}[1])=H^{0}(X,\mathscr{M})=0.
    \]
    Thus, it remains to show $(H^{0}(\mathscr{E})\otimes\mathscr{O}_{X})[1]$ and $\mathscr{E}$ satisfies the fourth assumption of Lemma \ref{theorem:wallCrossingOne}.
    
    For ease of notation, set
    \[
      (\beta_{0},\alpha_{0})=\left(\frac{\operatorname{ch}_{2}(\mathscr{E})}{\operatorname{deg}_{H}(\mathscr{E})},\frac{\operatorname{ch}_{2}(\mathscr{E})}{\operatorname{deg}_{H}(\mathscr{E})}\right)\in W( (H^{0}(\mathscr{E})\otimes\mathscr{O}_{X})[1],\mathscr{E}).
    \]
    By direct computation,
    \begin{equation}
      \label{equation:slopeInequality}
      \mu_{\alpha_{0}+\varepsilon,\beta_{0}}^{\mathrm{tilt}}( (H^{0}(\mathscr{E})\otimes\mathscr{O}_{X})[1])-\mu_{\alpha_{0}+\varepsilon,\beta_{0}}^{\mathrm{tilt}}(\mathscr{E})=\frac{\operatorname{deg}_{H}(\mathscr{E})^{2}\varepsilon (2\operatorname{ch}_{2}(\mathscr{E})+\operatorname{deg}_{H}(\mathscr{E})\varepsilon)}{2\operatorname{ch}_{2}(\mathscr{E})(\operatorname{deg}_{H}(\mathscr{E})^{2}-\operatorname{rank}(\mathscr{E})\operatorname{ch}_{2}(\mathscr{E}))}.
    \end{equation}
    By assumption, $\operatorname{ch}_{2}(\mathscr{E})>0$ and $\operatorname{deg}_{H}(\mathscr{E})>0$, so Equation \ref{equation:slopeInequality} is positive as long as
    \[
      \operatorname{deg}_{H}(\mathscr{E})^{2}-\operatorname{rank}(\mathscr{E})\operatorname{ch}_{2}(\mathscr{E})>0.
    \]
    However, since $\mathscr{E}$ is $(H,\frac{K_{X}}{2})$-Gieseker stable (a fortiori $\mu_{H}$-semistable by Lemma \ref{lemma:implicationsAmongCoherentStability}), this inequality follows from the assumption $\operatorname{ch}_{2}(\mathscr{E})>0$ and Bogomolov's Inequality (Lemma \ref{lemma:bogomolovInequality}).
    
    Hence, by Lemma \ref{theorem:wallCrossingOne}, $\mathscr{M}[1]$ is $\mu_{\alpha,\beta}^{\mathrm{tilt}}$-stable for $(\beta,\alpha)$ in the chamber directly above $W((H^{0}(\mathscr{E})\otimes\mathscr{O}_{X})[1],\mathscr{E})$, as desired.
  \end{proof}
  
  If we could show $W((H^{0}(\mathscr{E})\otimes\mathscr{O}_{X})[1],\mathscr{E})$ is the largest actual wall associated to $\mathscr{M}[1]$ then $\mathscr{M}$ would be $\mu_{H}$-stable by Lemma \ref{lemma:LazarsfeldStableAboveWall} and the Large Volume Limit.
  This technique is used in \cite{feyzbakhsh2022}*{Theorem 1.3} to show the kernel sheaf associated to an ample line bundle on a $K3$ surface is $\mu_{H}$-stable.
  However, Feyzbakhsh's proof fundamentally uses that $K3$ surfaces satisfy a stronger form of Bogomolov's inequality.
  Moreover, even with a stronger Bogomolov's inequality, Feyzbakhsh's argument does not generalize to higher rank bundles.
  
  Like $K3$ surfaces, Del Pezzo surfaces satisfy a stronger form of Bogomolov's inequality~\cite{douglas2006}*{Appendix A}.
  Even if we ignore some technicalities of Feyzbakhsh's method that involve extending tilt stability ``below" the $(\alpha,\beta)$-plane, this method still will not generalize to kernel sheaves associated to higher rank stable sheaves.
  However, assuming Theorem \ref{mainTheorem}, Proposition \ref{proposition:largestWallEquivalent} shows $W((H^{0}(\mathscr{E})\otimes\mathscr{O}_{X})[1],\mathscr{E})$ is the largest actual wall in the higher rank case.
  Either way, it unclear to the author how to show this conclusion directly for kernel sheaves associated to higher rank sheaves.
  
  Instead of showing $W((H^{0}(\mathscr{E})\otimes\mathscr{O}_{X})[1],\mathscr{E})$ is the largest actual wall directly, we use Lemma \ref{lemma:LazarsfeldStableAboveWall} to bound the Chern characters of a maximally destabilizing subsheaf of $\mathscr{M}$.
  In the case of Del Pezzo surfaces we use these Chern character bounds to prove Theorem \ref{mainTheorem}.
    
  \begin{lemma}
    \label{lemma:boundingSecondChern}
    Let $\mathscr{M}$ be the kernel sheaf associated $\mathscr{E}$ satisfying the assumptions of Lemma \ref{lemma:LazarsfeldStableAboveWall}.
    \begin{enumerate}
      \item{
        If $\mathscr{M}$ is not $\mu_{H}$-semistable and $0\to\mathscr{N}\to\mathscr{M}$ is a maximal $\mu_{H}$-destabilizing subsheaf then
        \[
          \frac{\operatorname{ch}_{2}(\mathscr{N})}{\operatorname{deg}_{H}(\mathscr{N})}\leq\frac{\operatorname{ch}_{2}(\mathscr{M})}{\operatorname{deg}_{H}(\mathscr{M})}.
        \]
      }
      \item{
        If $\mathscr{M}$ is $\mu_{H}$-semistable and $0\to\mathscr{N}\to\mathscr{M}$ is a proper subsheaf satisfying $\mu_{H}(\mathscr{N})=\mu_{H}(\mathscr{M})$ then
        \[
          \frac{\operatorname{ch}_{2}(\mathscr{N})}{\operatorname{deg}_{H}(\mathscr{N})}<\frac{\operatorname{ch}_{2}(\mathscr{M})}{\operatorname{deg}_{H}(\mathscr{M})}
        \]
      }
    \end{enumerate}
  \end{lemma}
  \begin{proof}
    Suppose $\mathscr{M}$ is the kernel sheaf associated to $\mathscr{E}$.
    \begin{enumerate}
      \item{
        Since $\overline{\Delta}_{H}(\mathscr{M}[1])>0$ there exists at most one wall associated to $\mathscr{M}[1]$ with endpoint $(0,0)$ (see discussion after Lemma \ref{lemma:wallsOfTiltStability}).
        Therefore, since $W(\mathscr{E},(H^{0}(\mathscr{E})\otimes\mathscr{O}_{X})[1])$ has endpoints $(0,0)$ and $(2\operatorname{ch}_{2}(\mathscr{E})/\operatorname{deg}_{H}(\mathscr{E}),0)$, by local finiteness of walls (Lemma \ref{lemma:stabilityIsConstantInChambers}) there is a line $\{0\}\times (0,\alpha_{0})\subseteq\mathbb{R}\times\mathbb{R}_{>0}$ contained in the chamber directly above the wall $W(\mathscr{E},(H^{0}(\mathscr{E})\otimes\mathscr{O}_{X})[1])$
        Thus, by Lemma \ref{lemma:LazarsfeldStableAboveWall}, there exists $\alpha_{0}>0$ such that $\mathscr{M}[1]$ is $\mu_{\alpha,0}^{\mathrm{tilt}}$-stable for all $\alpha\in (0,\alpha_{0})$.
        
        By Lemma \ref{lemma:maximalDestabilizingSubobjectIsInHeart}, $0\to\mathscr{N}[1]\to\mathscr{M}[1]$ is a subobject in $\operatorname{Coh}_{H}^{0}(X)$ and so
        \[
          \frac{\operatorname{ch}_{2}(\mathscr{N}[1])-\frac{\alpha^{2}}{2}\operatorname{rank}(\mathscr{N}[1])}{\operatorname{deg}_{H}(\mathscr{N}[1])} =\mu_{\alpha,0}^{\mathrm{tilt}}(\mathscr{N}[1]) <\mu_{\alpha,0}^{\mathrm{tilt}}(\mathscr{M}[1]) =\frac{\operatorname{ch}_{2}(\mathscr{M}[1])-\frac{\alpha^{2}}{2}\operatorname{rank}(\mathscr{M}[1])}{\operatorname{deg}_{H}(\mathscr{M}[1])}
        \]
        for all $\alpha\in(0,\alpha_{0})$.
        Taking the limit as $\alpha$ approaches $0$ gives
        \[
          \frac{\operatorname{ch}_{2}(\mathscr{N})}{\operatorname{deg}_{H}(\mathscr{N})}\leq\frac{\operatorname{ch}_{2}(\mathscr{M})}{\operatorname{deg}_{H}(\mathscr{M})},
        \]
        as claimed.
      }
      \item{
        By the same argument as part $1$, we find
        \[
          \frac{\operatorname{ch}_{2}(\mathscr{N})}{\operatorname{deg}_{H}(\mathscr{N})}\leq\frac{\operatorname{ch}_{2}(\mathscr{M})}{\operatorname{deg}_{H}(\mathscr{N})}.
        \]
        If we have equality, since $\mu_{H}(\mathscr{N})=\mu_{H}(\mathscr{M})$, $\mu_{\alpha,\beta}^{\mathrm{tilt}}(\mathscr{N}[1])=\mu_{\alpha,\beta}^{\mathrm{tilt}}(\mathscr{M}[1])$ for all $(\beta,\alpha)\in\mathbb{R}\times\mathbb{R}_{>0}$.
        In particular, $\mathscr{M}[1]$ is not $\mu_{\alpha,\beta}^{\mathrm{tilt}}$-stable for $\beta>\mu_{H}(\mathscr{M}[1])$.
        This contradicts Lemma \ref{lemma:LazarsfeldStableAboveWall}, so we must have
        \[
          \frac{\operatorname{ch}_{2}(\mathscr{N})}{\operatorname{deg}_{H}(\mathscr{N})}<\frac{\operatorname{ch}_{2}(\mathscr{M})}{\operatorname{deg}_{H}(\mathscr{M})}
        \]
        as claimed.
      }
    \end{enumerate}
  \end{proof}
  
  In the case of Del Pezzo surfaces, we can use the above bounds to prove $\mu_{H}$-stability of $\mathscr{M}$.
  
  \begin{theorem}
    \label{theorem:lazarsfeldMukaiOnDelPezzo}
    Assume $X$ is a smooth Del Pezzo surface over an algebraically closed field (not necessarily of characteristic $0$)---so $-K_{X}$ is ample.
    Let $\mathscr{E}$ be a globally generated, torsion-free, $(-K_{X},-\frac{K_{X}}{2})$-Gieseker stable sheaf on $X$ with associated kernel sheaf $\mathscr{M}$.
    If the following bounds are satisfied:
    \begin{itemize}
      \item{
        $0<\operatorname{deg}_{-K_{X}}(\mathscr{E})$,
      }
      \item{
        $0<\operatorname{ch}_{2}(\mathscr{E})$,
      }
      \item{
        $\displaystyle \overline{\Delta}_{-K_{X}}(\mathscr{E})<2\frac{\operatorname{ch}_{2}(\mathscr{E})}{\operatorname{deg}_{-K_{X}}(\mathscr{E})}+\frac{1}{\operatorname{rank}(\mathscr{E})^{2}}$
      }
    \end{itemize}
    then $\mathscr{M}$ is $\mu_{-K_{X}}$-stable.
  \end{theorem} 
  \begin{proof}
    Consider a maximal $\mu_{-K_{X}}$-destabilizing subsheaf $0\to\mathscr{N}\to\mathscr{M}$.
    In other words, if $\mathscr{M}$ is $\mu_{-K_{X}}$-semistable then $\mu_{-K_{X}}(\mathscr{N})=\mu_{-K_{X}}(\mathscr{M})$ and if $\mathscr{M}$ is no $\mu_{-K_{X}}$-semistable then, by \cite{rudakov1997}*{Lemma 3.2}, $\mu_{-K_{X}}(\mathscr{N})>\mu_{-K_{X}}(\mathscr{M})$.
    Since $X$ is a Del Pezzo surface, by the Hirzebruch-Riemann-Roch theorem,
    \[
      \frac{\chi(\mathscr{N})}{\operatorname{deg}_{-K_{X}}(\mathscr{N})}=\frac{\operatorname{ch}_{2}(\mathscr{N})}{\operatorname{deg}_{-K_{X}}(\mathscr{N})}-\frac{1}{2}+\frac{\operatorname{rank}(\mathscr{N})}{\operatorname{deg}_{-K_{X}}(\mathscr{N})}.
    \]
    By \cite{rekuski2023}*{Lemma 4.1}, $\operatorname{deg}_{-K_{X}}(\mathscr{N})<0$ and so the above equality is well-defined.
    Therefore, since $0\to\mathscr{N}\to\mathscr{M}$ is a maximal $\mu_{-K_{X}}$-destabilizing subsheaf, by Lemma \ref{lemma:boundingSecondChern},
    \[
      \frac{\chi(\mathscr{N})}{\operatorname{deg}_{-K_{X}}(\mathscr{N})}=\frac{\operatorname{ch}_{2}(\mathscr{N})}{\operatorname{deg}_{-K_{X}}(\mathscr{N})}-\frac{1}{2}+\frac{\operatorname{rank}(\mathscr{N})}{\operatorname{deg}_{-K_{X}}(\mathscr{N})}<\frac{\operatorname{ch}_{2}(\mathscr{M})}{\operatorname{deg}_{-K_{X}}(\mathscr{M})}-\frac{1}{2}+\frac{\operatorname{rank}(\mathscr{M})}{\operatorname{deg}_{-K_{X}}(\mathscr{M})}=\frac{\chi(\mathscr{M})}{\operatorname{deg}_{-K_{X}}(\mathscr{M})}.
    \]
    Moreover, since $X$ is a Del Pezzo surface,
    \begin{equation}
      \label{equ:firstBound}
      \frac{\chi(\mathscr{N})}{\operatorname{deg}_{-K_{X}}(\mathscr{N})}<\frac{\chi(\mathscr{M})}{\operatorname{deg}_{-K_{X}}(\mathscr{M})}=\frac{h^{1}(\mathscr{E})}{\operatorname{deg}_{-K_{X}}(\mathscr{M})}\leq 0.
    \end{equation}
    
    Showing $\chi(\mathscr{N})<0$ will lead to a contradiction.
    Since $\mathscr{M}$ is a kernel sheaf, $h^{0}(\mathscr{M})=0$ and so $h^{0}(\mathscr{N})=0$ as well.
    Therefore, it suffices to show $h^{2}(\mathscr{N})=0$.

    Since $\mathscr{N}$ is torsion-free, there is a natural injection $0\to\mathscr{N}\to\mathscr{N}^{\vee\vee}$ whose cokernel is supported in dimension $2$ so $h^{2}(\mathscr{N})=h^{2}(\mathscr{N}^{\vee\vee})$.
    Moreover, by \cite{hartshorne1980}*{Corollary 1.4}, $\mathscr{N}^{\vee}$ is locally free so, by Serre duality, $h^{2}(\mathscr{N})=h^{0}(\mathscr{N}^{\vee}\otimes\omega_{X})$ where $\omega_{X}$ is the canonical line bundle.
    By direct calculation,
    \begin{equation}
      \label{equation:boundOnMaximal}
      \mu_{-K_{X}}^{+}(\mathscr{N}^{\vee}\otimes\omega_{X})=-\mu_{-K_{X}}(\mathscr{N})+\operatorname{deg}_{-K_{X}}(\omega_{X})\leq-\mu_{-K_{X}}(\mathscr{M})+\operatorname{deg}_{-K_{X}}(\omega_{X}).
    \end{equation}
    Therefore, if $-\mu_{-K_{X}}+\operatorname{deg}_{-K_{X}}(\omega_{X})<0$ then $h^{2}(\mathscr{N})=0$.

    Since $X$ is a Del Pezzo $(-K_{X})^{2}\geq 3$, so
    \[
      \frac{\operatorname{deg}_{-K_{X}}(\mathscr{E})}{(-K_{X})^{2}}-\frac{\operatorname{deg}_{-K_{X}}(\mathscr{E})}{2}-h^{1}(\mathscr{E})\leq 0<\operatorname{ch}_{2}(\mathscr{E}).
    \]
    In other words, by the Hirzebruch-Riemann-Roch theorem
    \[
      \operatorname{deg}_{-K_{X}}(\mathscr{E})\leq (-K_{X})^{2}\left( h^{1}(\mathscr{E})+\frac{\operatorname{deg}_{-K_{X}}(\mathscr{E})}{2}+\operatorname{ch}_{2}(\mathscr{E})\right)=(-K_{X})^{2}(h^{0}(\mathscr{E})-\operatorname{rank}(\mathscr{E})).
    \]
    That is to say, $-\mu_{-K_{X}}(\mathscr{M})+\operatorname{deg}_{-K_{X}}(\omega_{X})\leq 0$.
    Hence, by Equation \ref{equation:boundOnMaximal} we find $\mu_{-K_{X}}^{+}(\mathscr{N}^{\vee}\otimes\omega_{X})<0$ and so
    \[
      h^{2}(\mathscr{N})=h^{0}(\mathscr{N}^{\vee}\otimes\omega_{X})=\operatorname{Hom}(\mathscr{O}_{X},\mathscr{N}^{\vee}\otimes\omega_{X})=0, 
    \]
    as claimed.
    
    In all, we have shown $\chi(\mathscr{N})=-h^{1}(\mathscr{N})\leq 0$.
    Hence, by Equation \ref{equ:firstBound},
    \[
      0\leq-\frac{h^{1}(\mathscr{N})}{\operatorname{deg}_{-K_{X}}(\mathscr{N})}=\frac{\chi(\mathscr{N})}{\operatorname{deg}_{-K_{X}}(\mathscr{N})}<0,
    \]
    a contradiction.
    Hence, $\mathscr{M}$ must be $\mu_{-K_{X}}$-stable, as desired.
  \end{proof}
  
  In the case of a torsion-free sheaf of rank $1$ on a Del Pezzo surface, Theorem \ref{theorem:lazarsfeldMukaiOnDelPezzo} slightly strengthens previous best known result~\cite{torres2022}*{Theorem 3.1}.
    
  Any sufficiently high twist of a torsion-free $\mu_{-K_{X}}$-stable sheaf satisfies the assumptions of Theorem \ref{theorem:lazarsfeldMukaiOnDelPezzo}:
  
  \begin{corollary}
    \label{corollary:effectiveVersion}
    Assume $X$ is a smooth Del Pezzo surface.
    Let $\operatorname{reg}_{-K_{X}}(\mathscr{E})$ be the Castelnuovo-Mumford regularity of $\mathscr{E}$ with respect to $-K_{X}$ (see \cite{lazarsfeld2004}*{Definition 1.8.3}) where $\mathscr{E}$ is a globally generated, torsion-free sheaf on $X$.
    If
    \[
      d\geq\max\left\{\sqrt{\frac{4\overline{\Delta}_{-K_{X}}(\mathscr{E})}{\operatorname{rank}(\mathscr{E})^{2}}+1}-\mu_{H}(\mathscr{E})+\frac{1}{2},\overline{\Delta}_{-K_{X}}(\mathscr{E})-\mu_{-K_{X}}(\mathscr{E}),\operatorname{reg}_{-K_{X}}(\mathscr{E})\right\}
    \]
    and $\mathscr{E}$ is $(-K_{X},d K_{X}-\frac{K_{X}}{2})$-Gieseker stable (e.g. $\mathscr{E}$ is $\mu_{-K_{X}}$-stable) then the kernel sheaf associated to $\mathscr{E}(-d K_{X})$ is $\mu_{H}$-stable.
  \end{corollary}
  \begin{proof}
    By Mumford's theorem~\cite{lazarsfeld2004}*{Theorem 1.8.3}, $\mathscr{E}(-dK_{X})$ is globally-generated, so $\mathscr{M}$ is well-defined.
    Moreover, by definition of the Castelnuovo-Mumford regularity, $h^{1}(\mathscr{E}(d)),h^{2}(\mathscr{E}(d))=0$.
    By the Hirzebruch-Riemann-Roch theorem and direct computation, $\mathscr{E}(-dK_{X})$ satisfies the assumptions of Theorem \ref{theorem:lazarsfeldMukaiOnDelPezzo}.
    The result follows.
  \end{proof}
  
  We discuss the assumed bounds of Theorem \ref{theorem:lazarsfeldMukaiOnDelPezzo}.
  \begin{remark}
    \label{example:needMoreThanJustPositivity}
    If $\mathscr{E}$ is globally generated then $\operatorname{deg}_{H}(\mathscr{E})\geq 0$.
    If $\operatorname{deg}_{H}(\mathscr{E})=0$ then $\mathscr{E}=H^{0}(\mathscr{E})\otimes\mathscr{O}_{X}$~\cite{rekuski2023}*{Lemma 3.9} so Theorem \ref{theorem:lazarsfeldMukaiOnDelPezzo} holds trivially.
    Therefore, we might as well assume $0<\operatorname{deg}_{H}(\mathscr{E})$.
    
    It is unclear to the author whether the assumption $\operatorname{ch}_{2}(\mathscr{E})>0$ is necessary.
    If $\mathscr{E}$ is globally generated then both $\operatorname{c}_{2}(\mathscr{E})\geq 0$ and $\operatorname{ch}_{1}(\mathscr{E})^{2}-\operatorname{c}_{2}(\mathscr{E})\geq 0$ where $\operatorname{c}_{2}$ is the second Chern \textit{class}, but neither of these inequalities imply $\operatorname{ch}_{2}(\mathscr{E})>0$.
    However, if $\operatorname{ch}_{2}(E)\leq 0$ then $W((H^{0}(\mathscr{E})\otimes\mathscr{O}_{X})[1],\mathscr{E})$ is not an actual wall associated to $\mathscr{M}[1]$, so the method of proof does not even apply.
    
    Last, the assumptions $0<\operatorname{deg}_{-K_{X}}(\mathscr{E})$ and $0<\operatorname{ch}_{2}(\mathscr{E})$ are not sufficient in Theorem \ref{theorem:lazarsfeldMukaiOnDelPezzo}.
    Consider the tangent bundle $\mathscr{T}_{\mathbb{P}^{2}}$ on $\mathbb{P}^{2}$ and let $H$ be the ample generator of the Picard group.
    Using the Euler sequence, there is a short exact sequence
    \[
      0\to\Omega_{\mathbb{P}^{2}}(1)^{\oplus 3}\to H^{0}(\mathscr{T}_{\mathbb{P}^{2}})\otimes\mathscr{O}_{\mathbb{P}^{2}}\to\mathscr{T}_{\mathbb{P}^{2}}\to 0.
    \]
    In particular, the kernel bundle associated to $\mathscr{T}_{\mathbb{P}^{2}}$ is not $\mu_{H}$-stable even though $\mathscr{T}_{\mathbb{P}^{2}}$ is $\mu_{H}$-stable with $\operatorname{deg}_{H}(\mathscr{T}_{\mathbb{P}^{2}}),\operatorname{ch}_{2}(\mathscr{T}_{\mathbb{P}^{2}})>0$~\cite{okonek1980}*{Chapter 2 Theorem 1.3.2}.
    However,
    \[
      2 \frac{\operatorname{ch}_{2}(\mathscr{E})}{\operatorname{deg}_{H}(\mathscr{E})}+\frac{1}{\operatorname{rank}(\mathscr{E})^{2}}=1+\frac{1}{4}<3=\overline{\Delta}_{H}(\mathscr{T}_{\mathbb{P}^{2}}),
    \]
    so $\mathscr{T}_{\mathbb{P}^{2}}$ does not satisfy the third positivity assumptions of Theorem \ref{theorem:lazarsfeldMukaiOnDelPezzo}.
  \end{remark}
  
  We end by discussing possible generalizations of Theorem \ref{theorem:lazarsfeldMukaiOnDelPezzo}.
  
  First, we expect Theorem \ref{theorem:lazarsfeldMukaiOnDelPezzo} generalizes (at least in characteristic $0$) to smooth, projective surfaces with
  \begin{itemize}
    \item{
      geometric genus $0$,
    }
    \item{
      irregularity $0$, and
    }
    \item{
      either $K_{X}=H$ is ample or $K_{X}$ is numerically trivial with $H$ an arbitrary ample divisor.
    }
  \end{itemize}
  With these assumptions, the argument of Theorem \ref{theorem:lazarsfeldMukaiOnDelPezzo} holds except, possibly, the vanishing $H^{2}(\mathscr{N})=0$.
  
  Second, we note $\mu_{H}$-stability of $\mathscr{M}$ is equivalent to $W((H^{0}(\mathscr{E})\otimes\mathscr{O}_{X})[1],\mathscr{E})$ being the largest actual wall associated to $\mathscr{M}[1]$.
  In view of Theorem \ref{theorem:lazarsfeldMukaiOnDelPezzo}, we find that $W((H^{0}(\mathscr{E})\otimes\mathscr{O}_{X})[1],\mathscr{E})$ is the largest wall in the case of Del Pezzo surfaces, but we expect this bound to generalize.
  The following result was pointed out to the author by Xuqiang Qin.
  \begin{proposition}
    \label{proposition:largestWallEquivalent}
    Suppose $X$ is a smooth projective surface satisfying Bogomolov's inequality with fixed ample divisor $H$.
    Let $\mathscr{E}$ be a globally generated, torsion-free $(H,-\frac{K_{X}}{2})$-Gieseker stable sheaf on $X$ with associated kernel sheaf $\mathscr{M}$.
    Further assume the following bounds are satisfied
    \begin{itemize}
      \item{
        $0<\operatorname{deg}_{H}(\mathscr{E})$,
      }
      \item{
        $0<\operatorname{ch}_{2}(\mathscr{E})$, and
      }
      \item{
        $\overline{\Delta}_{H}(\mathscr{E})\leq 2\frac{\operatorname{ch}_{2}(\mathscr{E})}{\operatorname{deg}_{H}(\mathscr{E})}+\frac{1}{\operatorname{rank}(\mathscr{E})^{2}}$.
      }
    \end{itemize}
    The kernel sheaf $\mathscr{M}$ is $\mu_{H}$-stable if and only if $W((H^{0}(\mathscr{E})\otimes\mathscr{O}_{X})[1],\mathscr{E})$ is the largest actual wall associated to $\mathscr{M}[1]$ in the $(\alpha,\beta)$-plane.
  \end{proposition}
  \begin{proof}
    Suppose $\mathscr{M}$ is $\mu_{H}$-stable.
    Suppose for contradiction $W((H^{0}(\mathscr{E})\otimes\mathscr{O}_{X})[1],\mathscr{E})$ is not the largest actual wall associated to $\mathscr{M}[1]$ in the $(\alpha,\beta)$-plane.
    Therefore, we can find a $\mu_{\alpha_{0},\beta_{0}}^{\mathrm{tilt}}$-destabilizing subobject $0\to N\to\mathscr{M}[1]$ for $(\beta_{0},\alpha_{0})$ above the wall $W((H^{0}(\mathscr{E})\otimes\mathscr{O}_{X})[1],\mathscr{E})$.
    Note the heart of a bounded $t$-structure $\operatorname{Coh}_{H}^{\beta}(X)$ is independent of $\alpha$, so $0\to N\to\mathscr{M}[1]$ is a subobject for all $\alpha>0$.
    With this in mind, set
    \[
      \mu(\alpha)=\mu_{\beta_{0},\alpha}^{\mathrm{tilt}}(N)-\mu_{\beta_{0},\alpha}^{\mathrm{tilt}}(\mathscr{M}[1]).
    \]
    Since $\beta_{0}$ is fixed, $\mu:\mathbb{R}_{>0}\to\mathbb{R}$ is a quadratic or constant polynomial in $\alpha$.
    Also, by definition of $N$, $\mu(\alpha_{0})\geq 0$.
    
    Furthermore, by Lemma \ref{lemma:LazarsfeldStableAboveWall}, $\mathscr{M}[1]$ is $\mu_{\alpha,\beta_{0}}^{\mathrm{tilt}}$-stable for some region directly above the wall $W((H^{0}(\mathscr{E})\otimes\mathscr{O}_{X})[1],\mathscr{E})$.
    Therefore, there exists $0<\alpha_{1}<\alpha_{0}$ such that $\mu(\alpha_{1})<0$.
    Moreover, since $\mathscr{M}$ is $\mu_{H}$-stable, by the Large Volume Limit, $\mathscr{M}[1]$ is $\mu_{\alpha,\beta_{0}}^{\mathrm{tilt}}$-stable for some $\alpha>\alpha_{0}$.
    In other words, there exists $\alpha_{2}>\alpha_{0}$ such that $\mu(\alpha_{2})<0$.
    In all, we have shown $\mu$ has an inflection point in the interval $(\alpha_{1},\alpha_{2})$.
    Since $\mu$ is a either a quadratic or constant polynomial, there is a unique inflection point in this interval $(\alpha_{1},\alpha_{2})$.
    However, by direct computation, the inflection point of $\mu$ is at $\alpha=0$.
    Since $\alpha_{1}>0$ we have a contradiction.
    Hence, $W((H^{0}(\mathscr{E})\otimes\mathscr{O}_{X})[1],\mathscr{E})$ is the largest wall associated to $\mathscr{M}[1]$, as claimed.    
    
    The converse direction follows by the Large Volume Limit and Lemma \ref{lemma:LazarsfeldStableAboveWall}.
  \end{proof}
  
  In view of Proposition \ref{proposition:largestWallEquivalent} and Theorem \ref{theorem:lazarsfeldMukaiOnDelPezzo}, the author is hopeful it is possible to $W((H^{0}(\mathscr{E})\otimes\mathscr{O}_{X})[1],\mathscr{E})$ is the largest actual wall associated to $\mathscr{M}[1]$ for general surface.
  However, this seems to be a well known and difficult problem: see Problem 2.1 at \href{http://aimpl.org/stabmoduli/2/}{http://aimpl.org/stabmoduli/2/} and \cite{macri2017}*{Question 7.5}.

\end{document}